\documentclass{birkjour}
 \newtheorem{thm}{Theorem}[section]
 \newtheorem{cor}[thm]{Corollary}
 \newtheorem{lem}[thm]{Lemma}
 \newtheorem{prop}[thm]{Proposition}
 \theoremstyle{definition}
 \newtheorem{defn}[thm]{Definition}
 \theoremstyle{remark}
 \newtheorem{rem}[thm]{Remark}
 
 \numberwithin{equation}{section}
\usepackage[colorlinks=true,linkcolor=red,citecolor=blue]{hyperref}
\begin{document}
%
%
%
%
%
%
%
%
%

\title[A study on Dunford-Pettis completely continuous  like operators]
 {A study on  Dunford-Pettis completely\\
 continuous  like operators\\}
\author {M. Alikhani}
\address{Department of Mathematics, University of Isfahan}
\email{ m.alikhani@sci.ui.ac.ir and m2020alikhani@yahoo.com}

\date{January 1, 2004}
\begin{abstract} In this article, the class  of all
Dunford-Pettis $ p $-convergent operators  and
$ p $-Dunford-Pettis relatively compact property  on Banach spaces are investigated.\ Moreover,
we give some conditions on Banach spaces $ X $ and $ Y $ such that the class of bounded linear operators from $  X$ to $ Y $ and some its subspaces have the $ p $-Dunford-Pettis relatively compact property.\ In addition,
if $ \Omega $ is a compact Hausdorff space, then we prove that dominated operators from the  space of all continuous
functions from $ K $ to Banach space $ X $ (in short $ C(\Omega,X) $) taking values in a Banach space with the $ p $-$ (DPrcP) $ are $ p $-convergent when $ X $ has the Dunford-Pettis property of order $  p.$\ Furthermore,
we show that
 if $  T:C(\Omega,X)\rightarrow Y $ is a strongly bounded operator with representing measure $ m:\Sigma\rightarrow L(X,Y) $ and  $ \hat{T}:B(\Omega,X)\rightarrow Y $ is its extension,
then  $ T$ is Dunford-Pettis $ p $-convergent if and only if  $ \hat{T}$ is Dunford-Pettis $ p $-convergent.\
\end{abstract}
\maketitle
\section{ Introduction:}

A bounded subset $K$ of a Banach space $  X$ is called Dunford-Pettis set, if any weakly null sequence $ (x_{n}) $ in $ X^{\ast} $ converges uniformly on $ K$ \cite{An}.\
It is  easy to verify  that the class of  Dunford-Pettis sets  strictly contains  the class of relatively compact sets.\ But in general the converse is not true.\
The concept of  Dunford-Pettis relatively compact property (briefly denoted by $(DPrcP) $) on Banach spaces was introduced by Emmanuele \cite{e1}.\ A Banach space $ X $ has the $(DPrcP), $ if every Dunford-Pettis subset of $ X$ is relatively compact.\ It is well known that  any
 dominated operator from $ C(\Omega, X) $  spaces taking values in
a Banach space with the $ (DPrcP) $ is completely continuous when  $ X $ has the Dunford-Pettis property (see \cite{e1}).\\
Wen and Chen \cite{w},  introduced the definition of a Dunford-Pettis completely continuous operator in order to
characterize the $ (DPrcP) $ on Banach spaces.\
A bounded linear operator $ T $  from a Banach space $  X$ to a Banach space $ Y $ is called Dunford-Pettis completely continuous, if it transforms  weakly null sequences which are Dunford-Pettis sets in $ X $ to norm-null sequences in $ Y. $\\
For more details on the rule of the $ (DPrcP) $ in different areas of Banach space theory, one can refer to \cite{An,g9,w}.\ Let us recall from  \cite{djt}, that a sequence $ (x_{n})_{n} $ in $ X $ is  called weakly $ p $-summable, if $ (x^{\ast}(x_{n}))_{n} \in \ell_{p}$ for each $ x^{\ast}\in X^{\ast}.$\ The set of all weakly $ p $-summable sequences in $ X $ is denoted  by $ \ell_{p} ^{w}(X).$\ The weakly $ \infty $-summable sequences are precisely the weakly null sequences.\
A sequence $(x_{n})_{n}$ in $ X $ is said to be weakly $ p $-convergent to $ x\in X$ if $ (x_{n} - x)_{n}\in \ell_{p} ^{w}(X).$ \
Recently,
the concepts of Dunford-Pettis $ p$-convergent operators  between Banach spaces and $ p $-Dunford-Pettis relatively compact property  on Banach spaces was introduced by Ghenciu \cite{g11} as follows:\
\begin{itemize}
\item  An  operator $T:X\rightarrow Y $ is called Dunford-Pettis $  p$-convergent, if it takes weakly $ p $-summable sequences which are  Dunford-Pettis sets into norm null sequences.\ The class of  Dunford-Pettis $ p $-convergent  operators  from  $ X $ into $ Y $ is denoted by $ DPC_{p}(X,Y). $\
\item  A Banach space $ X$ has the $ p $-Dunford-Pettis relatively compact property (briefly
denoted by $ p $-$ (DPrcP) $), whenever every  Dunford-Pettis and weakly $ p $-summable sequence $ (x_{n})_{n} $ in $ X $ is norm null.\
\end{itemize}
Motivated by the above works, the following interesting questions are posed in this area:\\
\begin{itemize}
\item Under which conditions on Banach spaces $ X$ and $ Y, $ every dominated
operator $ T : C(\Omega,X) \rightarrow Y $ is $ p $-convergent?\
\item Suppose that $ T:C(\Omega,X)\rightarrow Y $ is a strongly bounded operator with representing measure $ m:\Sigma\rightarrow L(X,Y) $ and $ \hat{T}:B(\Omega,X)\rightarrow Y $ is the restriction of $T^{**}$ to $B(\Omega, X),$\ then is $T$ Dunford-Pettis $ p$-convergent if and only if $\hat{T}$ is Dunford-Pettis $ p$-convergent?\
 \end{itemize}
These kind of researches have done by many authors for different operators, see
\cite{bc,bb,cs,c,g10,g22}.\\ Here, we try answer to the above questions.\
The article is organized as follows:\\
Section 2 of this paper provides a wide range of  definitions and concepts in Banach spaces.\ These concepts are mostly
well known, but we need them in the sequel.\\
Section 3 provides the background knowledge of  Dunford-Pettis $ p $-convergent operators
and the $ p $-$ (DPrcP) $ on Banach spaces.\
Here, we investigate some characterizations
of $ p $-$ (DPrcP) $ on Banach spaces.\  In addition,  we prove that if $ (X_{n}) _{n}$
a sequence
of Banach spaces with the $ p $-$ (DPrcP) ,$ then  the space $ (\sum_{n=1}^{\infty}\bigoplus X_{n})_{\ell_{1}} $
has the  same property.\
In the sequel,
we show that the space bounded linear operators and some certain
subspaces of this space have the $ p $-$ (DPrcP) $ under suitable conditions.\ Moreover, if $ E $ is a Banach lattice, then we give some sufficient conditions for which each  Dunford-Pettis $ p $-convergent
operator $ T:E\rightarrow X $ has an adjoint which is Dunford-Pettis $ p $-convergent.\ Finally, we state that any dominated operator from
$ C(\Omega, X) $ spaces taking values in a Banach space with the $ p $-$ (DPrcP) $ is $p  $-convergent,  when $ X $
has the Dunford-Pettis property of order $ p .$\\
In Section 4, we show that if $  T:C(\Omega,X)\rightarrow Y $ is  strongly bounded, then
$  T $ is Dunford-Pettis $ p $-convergent if and only if  $ \hat{T} $ is Dunford-Pettis $ p $-convergent.\ Moreover, we prove that  $ T ^{\ast}$ is Dunford-Pettis $ p $-convergent if and only if  $ \hat{T} ^{\ast}$ is Dunford-Pettis $ p $-convergent.\ Furthermore, if $ \Omega $ is a dispersed compact Hausdorff space and $  T:C(\Omega,X)\rightarrow Y $ is  strongly bounded, then we show that $ T^{\ast} $  is Dunford-Pettis $p$-convergent  if and only if  $ m(A)^{\ast}:Y^{\ast}\rightarrow X^{\ast} $ is Dunford-Pettis $ p $-convergent, for each $ A \in \Sigma .$\
Note that our results in this section are motivated by results in \cite{bc,cs,g10,g22}.\
\section{ Definitions and Notations:}
 Throughout this article
  $ X,Y $ and $  Z$ are arbitrary Banach spaces,
 $ 1\leq p < \infty ,$   except for the cases where we consider other assumptions.\ Also, we suppose $p^{\ast}$ is the H$\ddot{\mathrm{o}}$lder conjugate of $p;$ if $ p=1,~~ \ell_{p^{\ast}} $
 plays the role of $ c_{0} .$\ The unit coordinate vector in $ \ell_{p} $ (resp.\ $ c_{0} $ or  $\ell_{\infty} $) is denoted by $ e_{n}^{p} $ (resp.\ $ e_{n} $).\ We
denote the closed unit ball of $  X$ by $ B_{X} $ and the identity map on $  X$  by
$ id_{X} .$\
The space $ X $ embeds in $ Y, $ if $ X $ is isomorphic to a closed subspace of $Y$ (in short we denote $ X\hookrightarrow Y $).\ We denote two isometrically isomorphic spaces $ X $ and $ Y $ by
$ X\cong Y.$\  Also we use $ \langle  x,x^{\ast}\rangle $ or $ x^{\ast}(x) $
for the duality between $ x\in X $ and $ x^{\ast}\in X^{\ast}. $\  The space of all bounded $ \Sigma $-measurable functions on $ \Omega $ with separable range in $ X $  is denoted by $ B(\Omega,X). $\
We  denote
the unit ball of $ C(\Omega,X)$ and the unit
ball of $ B(\Omega, X) $ by $  B_{0} $ and $ B, $ respectively.\
For a bounded linear operator $ T : X \rightarrow Y, $ the adjoint of the operator $ T $
is denoted by $ T^{\ast}. $\  The
space of all bounded linear operators (resp. compact operators) from $ X $ into $  Y$ is
denoted by $ L(X,Y ) $ (resp. $ K(X, Y ) $) and the topological dual of $ X $ is denoted by
$ X^{\ast}. $\ The space of all $ w^{\ast} $-$ w$ continuous  operators
from $ X^{\ast} $  to  $ Y $ will denoted by $L_{w^{\ast}}(X^{\ast},Y)  .$\ The projective
tensor product of two Banach spaces $ X $ and$ Y $ will be denoted by $ X\widehat{\bigotimes}_{\pi}Y. $\

In what follows we introduce some notions  which will be used in the sequel.\
A bounded linear operator $ T $  from a Banach space $  X$ to a Banach space $ Y $ is called  completely continuous, if it transforms  weakly null sequences  in $ X $ to norm-null sequences in $ Y $ \cite{di1}.\ The class  of all completely continuous operators from
$ X $ to $  Y$  is denoted by $ CC(X,Y) .$\ A Banach space $ X $ is said to have the
Dunford-Pettis property  (in short $X $ has the $ (DPP)$), if for any Banach space $ Y $ every weakly compact
operator $ T:X\rightarrow Y $ is completely continuous \cite{di1}.\
A bounded subset $ K$ of $ X $ is relatively weakly $  p$-compact (resp. weakly $  p$-compact) if every sequence in $ K $ has a weakly $p $-convergent subsequence
with limit in $ X $ (resp. in $ K $) \cite{cs}.\
A bounded linear operator $ T : X \rightarrow  Y  $ is called weakly $ p $-compact, if $ T (B_{X}) $ is a relatively weakly $ p $-compact set in $ Y.$ \
 A sequence $(x_{n})_{n}$ in $ X $  is called weakly $  p$-Cauchy if   $ (x_{n_{k}}-x_{m_{k}})_{k} $
is weakly $ p $-summable for any increasing sequences $ (n_{k})_{k} $ and $ (m_{k})_{k} $ of positive
integers \cite{ccl}.\
A bounded linear operator $ T :X\rightarrow Y$  is called $ p $-convergent, if it takes weakly $ p $-summable sequences  in $ X $ into norm null sequences in $ Y $ \cite{cs}.\ We denote the class  of $ p $-convergent operators from $ X  $ into $  Y$  by $ C_{p}(X,Y) .$\ A Banach space $ X $ has the $ p $-Schur property ($ X\in(C_{p}) $), if $ id_{X}\in C_{p} (X,X).$\
We refer the reader for undefined terminologies to the
classical references \cite{AlbKal, du, di1}.\
\section{ Dunford-Pettis relatively compact property of order $ p $}
In this section, we obtain some characterizations of
a Banach space with the $ p $-$ (DPrcP ). $\ Moreover the stability of  $ p $-$ (DPrcP )$ for
some subspaces of bounded linear operators is investigated.\\ Our aim in this section is to obtain some suitable conditions on  $ X $ and
$  Y$ such that any  dominated operator $ T $ from
$ C(\Omega,X) $ into $ Y $ is $  p$-convergent.\\
Recall that, the authors in \cite{ce1,g8} by using Right topology, independently, proved  that a sequence  $ (x_{n})_{n} $  in
a Banach space $ X $ is Right null if and only if it is Dunford-Pettis and weakly
null.\ Also, they showed that a sequence $ (x_{n})_{n} $ in a Banach space $ X $ is Right
Cauchy if and only if it is Dunford-Pettis and weakly Cauchy.\\
 Inspired by the above works, for convenience,
we apply the notions  $  p$-Right null and $ p $-Right Cauchy sequences instead of weakly $ p $-summable and weakly $ p $-Cauchy sequences which are Dunford-Pettis sets, respectively.\

\begin{prop}\label{p1}\rm  Let $ T : X \rightarrow Y $  be a bounded linear operator.\ The following statements  are equivalent:\\
$ \rm{(i)} $ $T$ maps Dunford-Pettis  and weakly $ p $-compact subset of $ X $  onto relatively norm compact in $ Y. $\\
$ \rm{(ii)} $ $ T $ maps $ p $-Right null  sequences onto norm null sequences,\\
$ \rm{(iii)} $ $  T$ maps  $ p $-Right Cauchy sequences onto norm convergent sequences.
\end{prop}
\begin{proof}
(i) $ \Rightarrow $ (ii)
 Suppose that $ (x_{n})_{n} $  is a $ p $-Right null sequence in $ X $ and $ K:=\lbrace x_{n}: n\in\mathbb{N} \rbrace \cup \lbrace 0\rbrace. $\ It is clear that $ K $ is a Dunfotd-Pettis weakly $ p $-compact set in $ X. $\  By $ \rm{(i)} ,$ $T(K)$ is relatively norm compact in $ Y, $ and so $ ( T(x_{n}))_{n} $ is norm convergent to zero.\\
(ii) $ \Rightarrow $ (iii) Let
$ (x_{n})_{n} $ is a weakly $ p $-Right Cauchy sequence in $ X. $\ Therefore for any two subsequences $  (a_{n})_{n}$ and $ (b_{n})_{n} $ of $ (x_{n})_{n}, $
$(a_{n}-b_{n})_{n}$ is a $ p $-Right null sequence in $ X. $\ So, $ \rm{(ii)} $ implies that
 $ (T(a_{n})-T(b_{n}))_{n} $ is a norm null sequence in $ Y. $\ Hence, $ ( T(x_{n}))_{n} $ is norm convergent.\\
(iii) $ \Rightarrow $ (i)
Suppose that
$ K $ is a Dunford-Pettis weakly $ p $-compact subset of $  X.$\ Let $ (y_{n})_{n} $ be a sequence in $ T(K). $\ Therefore there is a sequence $ (x_{n})_{n}\subseteq K $ such that $ y_{n}=T(x_{n}) ,$ for each $ n\in \mathbb{N} .$\  Since $ K $ is a  weakly $ p $-compact set,
$ (x_{n})_{n} $ has a weakly $ p $-Cauchy subsequence.\ Without
loss of generality we can assume that $ (x_{n})_{n} $ is a $ p $-Right Cauchy sequence.\ Therefore by $ \rm{(iii)} ,$
$ (T(x_{n}))_{n} $ is norm-convergent
in $ Y.$\ Hence $ T(K) $ is relatively  norm compact.\
\end{proof}
As an immediate consequence of the  Proposition \ref{p1}, we can conclude that the following result:
 \begin{cor}\label{c1}  Suppose that $ X$ is a Banach space.\
The following assertions  are equivalent:\\
$ \rm{(i)} $   $ X $ has the $ p $-$(DPrcP) .$\\
$ \rm{(ii)} $  The identity operator  $ id_{X}:X\rightarrow X $ is Dunford-Pettis  $ p $-convergent.\\
 $ \rm{(iii)} $ Every $ p $-Right Cauchy sequence  in $ X $ is norm convergent.\\
  $ \rm{(iv)} $ Every  Dunford-Pettis  and weakly $ p $-compact subset of $ X $  is relatively norm compact.\
\end{cor}

\begin{rem}\label{r1}\rm
$ \rm{(i)} $  If $ X\in C_{p}, $ then $ X$ has the $ p $-$(DPrcP), $	but
in general the converse is not true.\ For example  for all $ p\geq 2,$ $ \ell_{2} $ has the $ p $-$ (DPrcP), $ while $ \ell_{2}\not\in C_{p}.$\\
 $ \rm{(ii)} $ It is clear that, if $ 1\leq r<q\leq \infty $ and $ X $ has
the $ q $-$(DPrcP), $ then $ X $ has the $ r $-$(DPrcP). $ In particular, for
$ 1\leq p<\infty ,$ if $ X $ has the $(DPrcP), $  then $ X $ has the $ p $-$(DPrcP). $\ But, the converse is not true.\ For example,
  the space $ L_{1}[0, 1] $ contains no copy of $ c_{0} .$\ Therefore $ L_{1}[0, 1] $ has the $ 1 $-Schur property {\rm (\cite[Corollary 2.9]{dm})}.\ Hence, $ L_{1}[0, 1] $ has the $ 1 $-$ (DPrcP),$ while, $ L_{1}[0, 1] $ does not have the $ (DPrcP). $\\
$ \rm{(iii)} $ It is easy to verify that
 $ (e_{n})_{n} $ is a $ p $-Right null sequence in $ c_{0} $ such that $ \Vert  e_{n} \Vert=1 $  for all $ n\in \mathbb{N} .$\ Hence  $ c_{0}$ does not have the $ p $-$ (DPrcP).$\\
 $ \rm{(iv)} $ There exists a Banach space $ X $ with the $ p $-$ (DPrcP)$ such that if $ Y $ is a closed subspace of it, then the quotient space $ \frac{X}{Y}$ does not have this property.\ For example $c_0$ does not have the $ p $-$ (DPrcP),$ while $ \ell_{1}$ has the $ p $-$ (DPrcP)$  and $c_0$ is isometrically isomorphic to a quotient of $ \ell_{1}$ {\rm (\cite[Corollary 2.3.2]{AlbKal})}.\
\end{rem}

\begin{thm}\label{t1}
 Suppose that $ X$ is a Banach space.\ The following assertions  are equivalent:\\
$ \rm{(i)} $  $ X $ has the $ p $-$ (DPrcP),$\\
$ \rm{(ii)} $ For each Banach space $ Y,~ $ $ DPC_{p}(X, Y ) = L(X, Y ), $\\
$ \rm{(iii)} $ $ X $ is the direct sum of two Banach spaces with the $ p $-$ (DPrcP).$
\end{thm}
\begin{proof}
 The equivalence of $ \rm{(i)} $ and $ \rm{(ii)} $ are clear.\ Hence,
we only prove the equivalence of $ \rm{(i)} $ and $ \rm{(iii)} $.\\
$\rm{(i)}\Rightarrow \rm{(iii)} $ Note that $ X=X\bigoplus \lbrace 0\rbrace. $\\
$\rm{(iii)}\Rightarrow \rm{(i)} $ Let $ X = Y \bigoplus Z $ such that $ Y $ and $ Z $ have the $ p $-$  (DPrcP).$\ Consider
the projections $ P_{1} : X \rightarrow Y  $ and $ P_{2} : X \rightarrow Z. $\  Let $ K $ be a  Dunford-Pettis and weakly $ p $-compact subset of $ X. $\ Clearly,  $ P_{1} (K)$ is a Dunford-Pettis and weakly $ p $-compact subset of $ Y,$ and so it
is a norm compact set in  $ Y. $\ Similarly $ P_{2} (K)$ is a norm compact set in $ Z. $\ Also, it is clear that any sequence $ (x_{n})_{n} \subseteq K $ can
be written as $ x_{n} = y_{n} + z_{n}, $ where $ y_{n}\in P_{1} (K) $ and  $ z_{n}\in P_{2} (K). $\ Thus, there are
subsequences $ (y_{n_{k}} )_{k} $ and $ (z_{n_{k}} )_{k} $ and  $ y \in P_{1} (K) $ and $ z\in P_{2} (K)$  such that  $ x_{n_{k}}=y_{n_{k}} +z_{n_{k}}\rightarrow y+z.$\ Since, $ K $ is a weakly $ p $-compact set, $ y+z\in K $
 and so, $  K$ is a  norm compact set.\
\end{proof}
A bilinear operator
$ \phi:X\times Y\rightarrow Z $
is called separately compact if for each fixed $ y \in Y , $ the linear operator $ T_{y}: X\rightarrow Z : x\mapsto \phi (x,y)$  and for each fixed $ x \in X , $ the linear operator $ T_{x}: Y\rightarrow Z : y\mapsto \phi (x,y) $
  are compact.\\
  By using the same
argument as in the proof of {\rm (\cite[Lemm 6.3.19]{z})}, we obtain the following result.
\begin{thm}\label{t2}
 If  every symmetric bilinear separately compact operator $ S:X\times X\rightarrow c_{0} $ is Dunford-Pettis  $ p $-convergent, then $ X $ has the  $ p $-$ (DPrcP) .$\
\end{thm}
\begin{proof}
Suppose that $ T \in L(X,c_{0}) $ is  arbitrary.\ Define $ T\bigotimes T:X\times X\rightarrow c_{0} $
 by $  (T\bigotimes T)(x,y)=T(x)T(y)$ (coordinatewise product of two sequences in $ c_{0} $).\  It is easy to verify that $  T\bigotimes T$ is bilinear and satisfies the following properties:\\
$ \rm{(i)} $  $  (T\bigotimes T) $ is symmetric, i.e., $ (T\bigotimes T)(x,y)=(T\bigotimes T)(y,x) .$\\
$ \rm{(ii)} $ $ (T\bigotimes T) $  is separately compact. Since for each fixed $ x \in X,  $ the operator
$ T_{x}:X\rightarrow c_{0}$ defined by $ T_{x} (y)=T(x)T(y)=(T\bigotimes T)(x,y)$
is bounded linear such that $ \Vert T_{x} \Vert\leq \Vert  T \Vert^{2} \Vert x \Vert.$\ Without
loss of generality that $ \Vert T \Vert=1. $\ Since $ T(x) $ is a sequence in $ c_{0}, $ for convenience we
agree to denote the $  m$-th term of $ T(x) $ by $  T(x)_{m}.$\ Hence, for $ \varepsilon>0, $
there exists $ n\in \mathbb{N} $such that $ \vert T(x)_{m}\vert  <\varepsilon  $ for all $ m\geq n. $\\ For convenience, we
denote   $ (T(x)_{1}, T(x)_{2}, \cdot\cdot\cdot, T(x)_{n}, 0, 0,\cdot\cdot\cdot) $ by $ T(x) (\leq n)$ and
  $(0, 0,\cdot\cdot\cdot , 0, T(x)_{n+1}, T(x)_{n+2},\cdot\cdot\cdot)  $ by $ T(x) (> n).$\
Therefore
$$ T_{x}(y)=T(x)T(y)=T(x) (\leq n)T(y)+T(x) (> n)T(y).$$
Obviously, the set $ \lbrace T(x) (\leq n)T(y): y\in X   \rbrace $  is contained in a finite-dimensional
subspace
$$ M:=\lbrace (\alpha_{i}) (\leq n) : (\alpha_{i})\in c_{0} \rbrace  $$
of $ c_{0}. $\ Hence,
$$ \lbrace T(x) (\leq n)T(y): y\in X   \rbrace\subseteq \Vert T(x) \Vert B_{M} ,$$
which is compact.\\
For each $ \varepsilon>0 $ there exists $ z_{1},\cdot\cdot\cdot, z_{j} \in \Vert T(x)\Vert B_{M} $  such that
 $$ \Vert T(x) \Vert B_{M}\subseteq \bigcup_{i=1}^{j} (z_{i}+\varepsilon B_{c_{0}})$$
So,
$$ T_{x}(B_{X}) =T(x)T(B_{X})\subseteq \bigcup_{i=1}^{j}(z_{i}+\varepsilon B_{c_{0}})+\varepsilon B_{c_{0}}\subseteq \bigcup_{i=1}^{j+1}(z_{i}+2\varepsilon B_{c_{0}})$$ with $ z_{j+1}=0. $
 Therefore, $  T_{x}(B_{X})$ is relatively compact.\ Hence for each fixed $x\in X ,$ the operator $ T_{x} $ is compact.\ Hence
 $  T\bigotimes T $
is separately
compact.\ Therefore, by the hypothesis $  T\bigotimes T $ is Dunford-Pettis $ p $-convergent.\ So, if $ (x_{n})_{n} $ is a $ p $-Right null sequence in $ X, $
then
$$ \lim_{n\rightarrow\infty}\Vert T(x_{n}) \Vert ^{2}=\lim_{n\rightarrow \infty}\sup_{i}(\vert (T(x_{n})_{i}) \vert^{2})=\lim_{n\rightarrow\infty}\Vert (T(x_{n}))^{2}  \Vert$$  $$~~~~~~~=\lim_{n\rightarrow\infty}\Vert(T\bigotimes T)(x_{n},x_{n})\Vert=0.$$
This shows that $ T $ is  Dunford-Pettis $ p $-convergent.\ Since $ T \in L(X,c_{0}) $  was arbitrarily
chosen it follows from Theorem  of \ref{t1} that $ X $ has $ p $-$ (DPrcP) .$
 \end{proof}
\begin{prop}\label{p2} If $ X $ has the $ p $-$ (DPrcP) ,$ then the following statements hold:\\
$ \rm{(i)} $ $ \displaystyle \lim_{n\rightarrow\infty} x_{n}^{\ast}(x_{n})=0,$  for every  $  p$-Right Cauchy sequence $ (x_{n})_{n} $ in $  X$ and every weakly null sequence $ (x^{\ast}_{n})_{n} $ in $  X^{\ast},$\\
$ \rm{(ii)} $ $ \displaystyle \lim_{n\rightarrow\infty} x_{n}^{\ast}(x_{n})=0,$ for every  $  p$-Right null sequence $ (x_{n})_{n} $ in $  X$ and every weakly null sequence $ (x^{\ast}_{n})_{n} $ in $  X^{\ast},$\\
$ \rm{(iii)} $ $ \displaystyle \lim_{n\rightarrow\infty} x_{n}^{\ast}(x_{n})=0,$ for every  $  p$-Right null sequence $ (x_{n})_{n} $ in $  X$ and every weakly Cauchy sequence $ (x^{\ast}_{n})_{n} $ in $  X^{\ast}.$\
\end{prop}
\begin{proof}
$ \rm{(i)} $ Let $ (x_{n})_{n} $ be a  $  p$-Right Cauchy sequence $ (x_{n})_{n} $ in $  X$ and $ (x^{\ast}_{n})_{n} $ be a weakly null sequence  in $  X^{\ast}.$\ Define a bounded linear operator $ T: X\rightarrow c_{0} $ by $ T(x)=(x^{\ast}_{n}(x))_{n} .$\ By Theorem \ref{t1}, $ T\in DPC_{p}(X,c_{0}). $\ Therefore, Proposition \ref{p1} implies that $ (T(x_{n}))_{n} $ converges to some $ \alpha=(\alpha_{n})_{n}\in c_{0} $ in norm.\ Hence, for every $ \varepsilon> 0 $ there exists a positive integer $ N_{1} $ such that $ \Vert  T(x_{n})-\alpha \Vert < \frac{\varepsilon}{2} $  for all $ n > N_{1}. $\  Since $ \alpha\in c_{0}, $ we choose another positive integer $ N_{2} $ such that $ \vert \alpha_{n} \vert < \frac{\varepsilon}{2} $  for all $ n> N_{2} .$\ Hence,  $ \vert    x_{n}^{\ast}(x_{n}) \vert <\varepsilon $ for all $ n > max\lbrace N_{1},N_{2}\rbrace. $\ Thus, $  \displaystyle\lim_{n\rightarrow\infty} x_{n}^{\ast}(x_{n})=0.$\\
$ \rm{(ii)} $  is trivial.\\
$ \rm{(iii)} $  Suppose there exists a  $ p $-Right null sequence $ (x_{n})_{n} $ in $ X $ and there exists a weakly Cauchy sequence $  (x_{n}^{\ast})_{n}$ in $ X^{\ast} $ such that $ \vert x_{n}^{\ast}(x_{n}) \vert>\varepsilon,$ for some $ \varepsilon>0 $ and all $ n\in\mathbb{N}.$\  Since $ (x_{n})_{n} $ is weakly  $ p $-summable and in particular weakly null, there exists a subsequence $ (x_{k_{n}})_{n} $ of $ (x_{n})_{n} $ such that $ \vert x_{n}^{\ast}(x_{k_{n}}) \vert <\frac{\varepsilon}{2} $ for all $ n\in \mathbb{N}. $\ Since $ (x^{\ast}_{n})_{n} $ is weakly Cauchy, we see that  $ (x^{\ast}_{k_{n}}-x^{\ast}_{n})_{n} $ is weakly null.\  Now, by $ \rm{(ii)}, $ we have $ \lim_{n\rightarrow\infty} (x^{\ast}_{k_{n}}-x^{\ast}_{n})(x_{k_{n}})=0.$\ This implies that  $ \vert  (x^{\ast}_{k_{n}}-x^{\ast}_{n})(x_{k_{n}})   \vert <\frac{\varepsilon}{3} $ for $ n $ large enough.\ But for such $ n$'s, we have $$ \varepsilon <\vert x^{\ast}_{k_{n}}(x_{k_{n}})  \vert \leq \vert ( x^{\ast}_{k_{n}} -x^{\ast}_{n})(x_{k_{n}})  \vert+ \vert x^{\ast}_{n}(x_{k_{n}}) \vert  <\frac{5\varepsilon}{6},$$  which is a contradiction.\
\end{proof}
A bounded linear operator $ T : X \rightarrow  Y  $ is called  strictly singular, if there is no infinite dimensional subspace $ Z \subseteq X $ such that $ T\vert_{Z} $ is an isomorphism onto its range \cite{AlbKal}.\
\begin{prop}\label{p3}
 Suppose that $ T \in DPC_{p}(X, Y ) $ is not strictly singular.\ Then, $ X $
and $ Y $ contain simultaneously some infinite dimensional closed subspaces with
the $ p $-$ (DPrcP). $\
\end{prop}
\begin{proof} Suppose that $ T $ has a bounded inverse on the closed infinite dimensional subspace $Z$ of $X.$\ If $ (x_{n})_{n} $ is a  $ p $-Right null sequence in $ Z, $ then $ (x_{n})_{n} $ is a $ p $-Right null sequence in $ X. $\ By the hypothesis, $ \Vert T(x_{n})\Vert\rightarrow 0 $ and so $ \Vert x_{n} \Vert\rightarrow 0.$\ Hence, $  Z$ has the $ p $-$ (DPrcP). $\ Similarly, we can see that  $ T(Z)  $ has the  $ p $-$ (DPrcP). $\
\end{proof}
Suppose that $ (X_{n})_{n\in \mathbb{N}}  $ is a sequence of Banach spaces.\ The space of all vector-valued sequences $ (\displaystyle\sum_{n=1}^{\infty}\oplus X_{n})_{\ell_{p}} $ is called, the infinite direct sum of $ X_{n} $ in the sense of $ \ell_{p}, $ consisting of all sequences $ x=(x_{n})_{n} $ with values in $X_{n} $ such that $ \Vert x \Vert_{p}=(\displaystyle\sum_{n=1}^{\infty}\Vert x_{n}\Vert^{p} )^{\frac{1}{p}}<\infty.$\
\begin{prop}\label{p4}
Let  $ (X_{n})_{n\in \mathbb{N}} $ be a family of Banach space.\ Then, $ X_{n}$ has the $ p $-$ (DPrcP) $ for all $ n\in\mathbb{N} $ if and only if  $ (\displaystyle\sum_{n=1}^{\infty}\oplus X_{n})_{\ell_{1}} $ has the same property.
\end{prop}
\begin{proof}
It is clear that if $  X=(\displaystyle\sum_{n=1}^{\infty}\oplus X_{n})_{\ell_{1}}$ has the $ p $-$ (DPrcP) ,$ then  every closed subspace of $ X $ has the $ p $-$ (DPrcP) . $\ Hence $ X_{n}$ has the $ p $-$ (DPrcP) $ for all $ n\in\mathbb{N} .$\ Now,  suppose that  $ (x_{n})_{n} $ is a $ p $-Right null sequence in $ X ,$
 where $ x_{n}=(b_{n,k})_{k\in \mathbb{N}} .$\ It is clear  that $  (b_{n,k})_{k\in \mathbb{N}}$ is a $ p $-Right null sequence in $ X_{k} $
 for all $ k\in \mathbb{N}. $\ Since $ X_{k} $ has the $ p $-$ (DPrcP) ,$  $\Vert  b_{n,k} \Vert_{X_{k}}\rightarrow 0  $ as $ n\rightarrow\infty $ for all $ k\in \mathbb{N} .$\
Using the techniques which used in {\rm (\cite[Lemma, page 31]{d})},
we can conclude that the sum $\Vert x_{n}\Vert_{1}=\displaystyle\sum_{k=1}^{\infty}\Vert b_{n,k} \Vert_{X_{k}}$ is converges uniformly in $ n. $\ Hence,  $ \displaystyle\lim_{n\rightarrow\infty}\Vert x_{n} \Vert_{1}=\displaystyle\lim_{n\rightarrow\infty}\sum_{k=1}^{\infty}\Vert b_{n,k} \Vert_{X_{k}}=\displaystyle\sum_{k=1}^{\infty}\lim_{n\rightarrow\infty}\Vert b_{n,k} \Vert_{X_{k}}=0 .$\
\end{proof}

If
$ \mathcal{M} $ is a closed subspace
of operator ideal $\mathcal{U}(X,Y) ,$ then for arbitrary elements $ x\in X $ and  $ y^{\ast}\in Y^{\ast}, $ the evaluation operators $ \varphi_{x} : \mathcal{M}\rightarrow Y$ and
$\psi_{y^{\ast}} : \mathcal{M} \rightarrow X^{\ast} $ on $
\mathcal{M} $  are defined by $ \varphi_{x}(T) = T(x)$ and $
\psi_{y^{\ast}}(T) = T^{\ast}(y^{\ast}) $ for $ T \in \mathcal{M}$ (see \cite{MZ}).\\
The following result shows that the Dunford-Pettis $ p$-convergent of all evaluation
operators is a necessary condition for
the $ p $-$ (DPrcP)$ of  closed subspace $\mathcal{M} \subseteq\mathcal{U}(X,Y) .$
\begin{cor}\label{c2}
 If  $\mathcal{M} $ is a closed subspace of operator ideal $ \mathcal{U}(X,Y)$ that has the $ p $-$ (DPrcP),$ then all of the evaluation operators $ \varphi_{x} :\mathcal{M} \rightarrow Y$ and $\psi_{y^{\ast}} : \mathcal{M} \rightarrow X^{\ast} $ are Dunford-Pettis $ p$-convergent.
\end{cor}

\begin{thm}\label{t3}
Let $ X $ and $  Y$ be two Banach spaces such that $ Y $ has the Schur property.\ If $ \mathcal{M} $ is a closed subspace of $  \mathcal{U}(X,Y)$ such that each evaluation operator $  \psi_{y^{\ast}}$ is Dunford-Pettis $ p $-convergent on
$ \mathcal{M}, $ then $ \mathcal{M} $  has the $ p $-$  (DPrcP).$
\end{thm}
\begin{proof}
 Suppose that $ \mathcal{M} $ does not have the $ p $-$  (DPrcP).$\ Therefore, there is a  $ p $-Right null sequence
$ (T_{n})_{n} $ in $ \mathcal{M} $ such that $ \Vert T_{n} \Vert \geq \varepsilon $ for all positive integer $  n$ and some $ \varepsilon >0. $\ We can choose a sequence $ (x_{n})_{n} $
in $ B_{X} $ such that $ \Vert T_{n}(x_{n}) \Vert \geq \varepsilon.$\ In addition, for each $ y^{\ast}\in Y^{\ast}, $ the evaluation operator is $  \psi_{y^{\ast}}$ is Dunford-Pettis  $ p $-convergent.\
So,
$ \vert \langle  y^{\ast} , T_{n}(x_{n}) \rangle\vert \leq \Vert T^{\ast}_{n}(y^{\ast})   \Vert \Vert x_{n} \Vert\rightarrow 0. $\
 Hence $ (T_{n}(x_{n}))_{n} $ is weakly null in $ Y, $
and so is norm null, which is a contradiction.\
\end{proof}
\begin{cor}\label{c3} {\rm (\cite[Theorem 2.3]{w})},
Let $ X $ and $  Y$ be two Banach spaces such that $ Y $ has the Schur property.\ If
$ \mathcal{M} $ is a closed subspace of $  \mathcal{U}(X,Y)$ such that each evaluation operator $  \psi_{y^{\ast}}$ is Dunford-Pettis  completely continuous on
$ \mathcal{M}, $ then $ \mathcal{M} $  has the $  (DPrcP).$
\end{cor}
\begin{cor}\label{c4}
 If $ Y $ has the Schur property and
$ \mathcal{M} $ is a closed subspace of $ L(X, Y ) $ such that each evaluation operators $ \psi_{y^{\ast}} $ is  Dunford-Pettis $ p $-convergent  on $ \mathcal{M}, $ then
$ \mathcal{M}$ has the $ p $-$ (DPrcP). $
\end{cor}
By a similar method, we obtain a sufficient condition for the $ p $-$ (DPrcP) $ of closed subspaces of $ L_{w^{\ast}}(X^{\ast},Y).$\
Since the proof of the following result is similar to the proof of Theorem \ref{t3}, we omit its proof.
 \begin{thm}\label{t4} If $ X $ has the Schur property and $ \mathcal{M} $ is a closed subspace of $ L_{w^{\ast}}(X^{\ast},Y)$ such that each evaluation operators  $\psi_{y^{\ast}} $ is Dunford-Pettis  $ p $-convergent on $ \mathcal{M}, $
then  $ \mathcal{M} $ has the $ p $-$ (DPrcP) .$
 \end{thm}
\begin{cor}\label{c5} The following statements  hold:\\
 $ \rm{(i)} $ If $ X^{\ast} $ has the $ p $-$(DPrcP) $ and $ Y  $ has the Schur property, then $  L(X, Y ) $ has the $ p $-$ (DPrcP).$ \\
$ \rm{(ii)} $ If $ X $ has the $  p $-$ (DPrcP) $ and $ Y  $ has the Schur property, then $ L_{w^{\ast}}(X^{\ast}, Y ) $ has the   $ p $-$ (DPrcP). $\\
 $ \rm{(iii)} $  If $ X^{\ast} $ has the $ p $-$ (DPrcP), $ then $\ell_{1}^{w}  (X^{\ast})$
 has the same property.\
\end{cor}

A Banach space $  X$ has the Dunford-Pettis property of order $p$ (in short  $ (DPP_{p})$), if every weakly compact operator
$ T: X\rightarrow Y $ is $ p $-convergent, for any Banach space $ Y$ \cite{cs}.\
\begin{rem}\label{r2}\rm Every $p$-convergent operator is Dunford-Pettis  $ p $-convergent, but in general the converse is not true.\ For example, the identity operator $ id_{\ell_{2}}:\ell_{2}\rightarrow\ell_{2} $ is weakly compact and so is Dunford-Pettis  $ 2 $-convergent, while it is not $ 2$-convergent.\ Ghenciu \cite{g9}, gave a characterization of those Banach spaces in which the converse of the above assertion holds.\ In fact, she showed that,  a  Banach space $ X $  has the  $ (DPP_{p}) $ if and only if   $ C_{p}(X,Y)=DPC_{p}(X,Y) ,$ for each Banach space $ Y. $\
\end{rem}
 Let  $ \mathcal{M} $ be a bounded subspace of $ \mathcal{U}(X,Y). $\ The point evaluation sets related to $ x \in X $ and $ y^{\ast} \in Y^{\ast} $ are the images of the closed
unit ball $ B_{\mathcal{M}} $ of $ \mathcal{M}, $ under the evaluation operators $ \phi_{x} $ and $\psi_{y^{\ast}}  $ are denoted by $\mathcal{M}_{1}(x)$ and $
\widetilde{\mathcal{M}_1}(y^{\ast})$  respectively \cite{MZ}.\
 Let us recall from \cite{ccl1}, that a bounded subset $ K $ of $ X $ is a $ p $-$ (V^{\ast} ) $ set, if
$  \displaystyle\lim_{n\rightarrow\infty} \displaystyle\sup_{x\in K}\vert x^{\ast}_{n}(x) \vert =0,$
for every weakly $ p $-summable sequence $ (x^{\ast}_{n})_{n}$ in $X^{\ast}.$\\

In the following a necessary condition for the  $ p $-$(DPrcP) $ of the dual of closed subspace $ \mathcal{M}\subseteq \mathcal{U}(X,Y) $ is given.
\begin{thm}\label{t5} Suppose that $ X^{\ast\ast}$ and $ Y^{\ast}  $ have the $  (DPP_{p}). $\ If $\mathcal{M}$  is a  closed subspace  of $ \mathcal{U}(X,Y) $ such that $ \mathcal{M}^{\ast} $ has the $ p $-$(DPrcP), $ then of all the point evaluations $ \mathcal{M}_{1}(x) $ and $\widetilde{\mathcal{M}} _{1}(y^{\ast})$ are $ p $-$ (V^{\ast}) $ sets in $ Y $ and $ X^{\ast} $ respectively.
\end{thm}
\begin{proof} Since $ \mathcal{M}^{\ast} $ has the  $ p $-$( DPrcP) ,$
 Theorem
\ref{t1}, implies that the  operators $ \varphi^{\ast}_{x}
 $ and $
\psi^{\ast}_{y^{\ast}}$  are  Dunford-Pettis $ p $-convergent.\ On the other hand, $ X^{\ast\ast}$ and $ Y^{\ast} $ have the $  (DPP_{p}). $\ Hence, $  \varphi^{\ast}_{x} $ and $\psi^{\ast}_{y^{\ast}}  $
are $ p $-convergent (see  {\rm (\cite[Theorem 3.18]{g11})}).\ Therefore, if
$(y^{\ast}_{n})_{n} $ is a weakly $ p $-summable sequence in $ Y^{\ast},$ then we have:\
 $$\lim_{n\rightarrow\infty}\sup\lbrace \vert y^{\ast}_{n}(T(x))\vert: T\in B_{\mathcal{M}} \rbrace $$ $$~~~~~~~~~~~=\lim_{n\rightarrow\infty} \sup \lbrace \vert \varphi^{\ast}_{x}(y^{\ast}_{n})(T)\vert:T\in B_{\mathcal{M}}\rbrace=\lim_{n\rightarrow\infty} \Vert \varphi^{\ast}_{x}(y^{\ast}_{n})\Vert=0 ,$$ for all $ x \in X. $\ Hence  $ (y^{\ast}_{n})_{n} $
 converges uniformly on $ \mathcal{M}_{1} (x).$\ This shows that $ \mathcal{M}_{1} (x)$ is a $ p $-$ (V^{\ast}) $ set in $Y,$  for all $ x\in X.$\ A similar proof shows that $ {\widetilde{\mathcal{M}}}_{1}(y^{\ast}) $ is a $ p $-$ (V^{\ast}) $ set in $ X^{\ast}, $ for all $y^{\ast} \in Y^{\ast}.$
 \end{proof}
 \begin{cor}\label{c6}  {\rm (\cite[Theorem 2.2]{w})} Suppose that $ X^{\ast\ast}$ and $ Y^{\ast}  $ have the $  (DPP). $\ If $  \mathcal{M}$ is a  closed subspace  of $ \mathcal{U}(X,Y) $ such that $ \mathcal{M}^{\ast} $ has the $(DPrcP), $ then of all the point evaluations $ \mathcal{M}_{1}(x) $ and $\widetilde{\mathcal{M}} _{1}(y^{\ast})$ are Dunford-Pettis sets in $ Y $ and $ X^{\ast} $ respectively.
\end{cor}
Let us recall from \cite{di1}, that a bounded subset $ K $ of $ X^{\ast} $ is called an $( L )$ set, if each weakly null sequence $ (x_{n})_{n} $ in $ X$ tends to $ 0 $ uniformly on $ K.$\
  \begin{thm}\label{t6} Suppose that $ L(X,Y )=K(X,Y) .$\ If
$ X^{\ast}$ has the $ p $-$(DPrcP) $ and $ Y $ has the  $ (DPrcP),$  then $ L(X,Y) $  has the $ p $-$(DPrcP) .$
\end{thm}
\begin{proof}
Suppose that $ (T_{n})_{n} $ is a $ p $-Right null sequence in $ L(X,Y) $
so that $ \Vert T_{n} \Vert =1$ for each $ n \in \mathbb{N}. $\ Let $ (y^{\ast}_{n})_{n} $
 be a sequence in $ B_{Y^{\ast}} $ and $ (x_{n})_{n} $ be a
sequence in $ B_{X} $ so that $ y^{\ast}_{n}(T_{n}(x_{n}))>\frac{1}{2} $
 for each $ n \in \mathbb{N}. $\ Since $ X^{\ast} $ has the  $ p $-$ (DPrcP), $  for each $ y^{\ast}\in Y^{\ast}, $
 the evaluation operator  $  \psi_{y^{\ast}}:L(X,Y)\rightarrow X^{\ast}$ is Dunford-Pettis $ p $-convergent.\
Therefore,
 \begin{center}
$ \vert \langle  y^{\ast} , T_{n}(x_{n}) \rangle\vert \leq \Vert T^{\ast}_{n}(y^{\ast})   \Vert \Vert x_{n} \Vert\rightarrow 0. $\
\end{center}
 So, $ (T_{n}(x_{n}))_{n} $ is weakly null in $ Y. $\ Now, we claim that $ \lbrace T_{n}(x_{n}):n\in \mathbb{N} \rbrace $ is a  Dunford-Pettis set in $ Y. $\ Let
 $ (z^{\ast}_{n})_{n} $ be a weakly null sequence in $ Y^{\ast}, $ and let
$ T \in  (X\widehat{\bigotimes}_{\pi}  Y^{\ast})^{\ast}.$\ It is well known that $ (X\widehat{\bigotimes}_{\pi}  Y^{\ast})^{\ast} \cong  L(X,Y^{\ast\ast}) $ (see \cite{du}, page 230).\
Therefore, by the hypothesis $ T^{\ast}_{\vert_{Y^{\ast}}} $
 is a compact operator.\ Hence, $ \lbrace T^{\ast}(z^{\ast}_{n}) :n\in  \mathbb{N} \rbrace $
   is relatively compact.\ Thus,
 $ \vert \langle x_{n}\otimes z_{n}^{\ast},T \rangle\vert\leq \Vert T^{\ast}(z^{\ast}_{n}) \Vert\rightarrow 0 ,$ and so
 $ (x_{n}\otimes z_{n}^{\ast})_{n} $
is weakly null in $ X\widehat{\bigotimes}_{\pi}  Y^{\ast}. $\
Since, $ L(X,Y) $ embeds isometrically in $ L(X,Y^{\ast\ast}), $ $ (T_{n})_{n} $ is a Dunford-Pettis sequence in $ L(X,Y^{\ast\ast}) .$\ But, a Dunford-Pettis subset of a dual space is necessarily an
$ (L) $-subset of the dual space.\ Therefore, $ z^{\ast}_{n}(T_{n}(x_{n})) \rightarrow 0.$\
 Thus $ (T_{n}(x_{n}))_{n} $ is a Dunford-Pettis and weakly null
 sequence in $ Y. $\ Hence, $ \Vert T_{n}(x_{n})\Vert\rightarrow 0, $ which is a  contradiction.\
\end{proof}
Let us recall from \cite{W} that, a norm $\Vert .\Vert  $ of a Banach lattice $ E $ is order continuous if for each net $ (x_{\alpha}) $
such that $ x_{\alpha}\downarrow 0 $ in $ E, $ the net $ (x_{\alpha}) $  is norm converges to $ 0 ,$ where the notation
$ x_{\alpha}\downarrow 0 $ means that the sequence $ (x_{\alpha}) $  is decreasing, its infimum exists and $\inf (x_{\alpha})=0.$ \

\begin{thm}\label{t7} Let  $E$ be a Banach lattice and  $ T:E\rightarrow X $ be a Dunford-Pettis $ p $-convergent operator.\
If the adjoint of
 $ T $ is  Dunford-Pettis $ p $-convergent, then one of the following assertions holds.\\
$\rm{(i)}$ The norm of $ E^{\ast} $ is order continuous.\\
$\rm{(ii)}$ $X^{\ast}$ has the $ p $-$ (DPrcP). $
\end{thm}
\begin{proof}
Assume that norm $ E^{\ast} $ is not order continuous and $ X^{\ast}$ does not have the $ p $-$ (DPrcP). $\
Since norm $ E^{\ast}$ is not order continuous, there exist a
sublattice $ {M}$ of $ E,$ such that it is isomorphic to $
\ell_{1}$ and a positive projection $ P:E\rightarrow \ell_{1}$
(see Theorem 1 in \cite{W}).\ Since $ X^{\ast}$ does not have the $ p $-$ (DPrcP), $ there exists a $p$-Right null and  normalized sequence $ (y^{\ast}_{n})_{n}$ in $X^{\ast}. $\ Therefore, there exit a sequence $ (y_{n})_{n} $ in $ X $
with $ \Vert y_{n}\Vert \leq 1 $ and an $ \varepsilon_{0}>0 $ such that $ \vert y^{\ast}_{n} (y_{n})\vert\geq \varepsilon_{0} $
for all $ n\in\mathbb{N}.$\ Suppose that $ T = S\circ P ,$ where the operator $ S:\ell_{1}\rightarrow X $ is defined by
$ S(\alpha_{n})=\displaystyle \sum_{n}\alpha_{n}y_{n}. $\
Since $ \ell_{1}$  is  $ p $-Schur space,  the operator $ T $ is Dunford-Pettis $ p $-convergent.\ Now, we claim that $ T^{\ast} $ is not a Dunford-Pettis $ p $-convergent operator.\ As the operator
$P$ is surjective, there exits $ \delta> 0 $ such that
$\delta B_{\ell_{1}}\subset P(B_{E}).$\ So, we have
$$ \Vert T^{\ast}(y^{\ast}_{n})\Vert =\sup_{x\in B_{E}}\vert T^{\ast}(y^{\ast}_{n})(x)\vert=\sup_{x\in B_{E}}\vert
y^{\ast}_{n}(T(x))\vert =$$
$$\sup_{x\in B_{E}}\vert
y^{\ast}_{n}(S(P(x))\vert\geq\delta .\vert
y^{\ast}_{n}(S(e^{1}_{n}))\vert \geq \delta
\varepsilon_{0},$$
where $ (e^{1}_{n})_{n} $ is the standard canonical basic of $ \ell_{1}.$\
Therefore, $ T^{\ast} $ is not a Dunford-Pettis $ p $-convergent operator, which
is a contradiction.
\end{proof}

Let us recall from \cite{e1}, that a bounded linear operator $ T: C(\Omega,X)\rightarrow Y$ is dominated, if there exists a positive linear functional $ L $ on $  C(\Omega)^{\ast}$ such that $ \Vert T(f) \Vert\leq L(\Vert  f\Vert) , ~~~~ f\in C(\Omega,X) . $\\
By a similar technique we obtain the following result which is the $  p$-version of {\rm (\cite[Theorem 11]{e1})}.\
\begin{thm}\label{t8}
Suppose  that $ Y$ has the $ p $-$ (DPrcP)$ and  $ K $ is a compact Hausdorff space.\ If $ X $ has the $ (DPP_{p}), $ then any dominated operator $ T $ from $ C (\Omega, X) $ into $ Y $ is  $ p $-convergent.
\end{thm}
\begin{proof}
Suppose that $ T:C(\Omega,X) \rightarrow Y$ is an arbitrary dominated operator.\ By
Theorem 5 in Chapter $ \rm{III} $ of \cite{di}, there is a function $ G $ from $\Omega $ into $ L(X, Y^{\ast\ast}) $ such that\\
$\rm{( i)} $ $ \Vert G(t)\Vert =1 ~\mu. a.e.  $ in $ \Omega. $ i.e.; $ \mu(\lbrace t\in \Omega : \Vert G(t)\Vert \not =1\rbrace)=0. $\\
$\rm{( ii)} $ For each $ y^{\ast}\in Y^{\ast} $ and $ f \in C (\Omega, X), $ the
function $ y^{\ast}( G(.)f(.)) $ is $ \mu
$-integrable and moreover
\begin{center}
$   y^{\ast}(T(f)) =\int_{\Omega}
y^{\ast}( G(t)f(t)) d\mu ~~ $ for $ ~~f\in C(\Omega,X). $
\end{center}
Where $ \mu $ is the least regular Borel measure dominating $ T.$\ Consider a  weakly  $ p $-summable sequence  $ (f_{n})_{n} $ in $ C(\Omega, X) .$\ Since
continuous linear images of weakly $ p$-summable sequences are weakly $ p$-summable sequences, $ (T(f_{n}))_{n} $ is a weakly $ p$-summable sequence in $ Y. $\ Now, we show that $ \lbrace T(f_{n})) :n\in\mathbb{N}\rbrace $ is a Dunford-Pettis set in $ Y. $\ For this purpose, we consider a weak null sequence
 $ (y_{n}^{\ast})_{n}$ in $Y^{\ast}. $\ It is not difficult to show that,
 for each $t \in \Omega,~  (G^{\ast}(t) y^{\ast}_{n})_{n} $
 is weakly null in $ X^{\ast} $ and
  $ (f_{n}(t))_{n} $ is a weakly $ p $-summable sequence in $ X. $\
 Since $ X $ has the $ (DPP_{p}), $ we have :
\begin{center}
$  y_{n}^{\ast}( (G(t) f_{n}(t)) )= G^{\ast}(t) y_{n}^{\ast}( f_{n}(t)) \rightarrow 0. $
\end{center}
 Moreover, there exists a constant $ M>0 $ such that
 $ \vert   y_{n}^{\ast}(G(t)f_{n}(t))\vert\leq M  $ for all $ t\in \Omega $ and $ n\in \mathbb{N}. $\ The Lebesgue dominated convergent Theorem, implies that:
 \begin{center}
 $ \displaystyle\lim_{n\rightarrow\infty} y_{n}^{\ast}(T(f_{n})) =\lim_{n\rightarrow\infty}\int_{\Omega} y_{n}^{\ast}( G(t)f_{n}(t) )d\mu =0. $
 \end{center}
Therefore,  $ \lbrace T(f_{n}) :n\in\mathbb{N}\rbrace
$ is a Dunford-Pettis set in $  Y$\ {\rm (\cite[Theorem 1]{An})}.\ Hence, $  (T(f_{n}))_{n}  $ is a $ p $-Right null
sequence in $ Y, $ and so  $ \Vert T(f_{n}) \Vert\rightarrow 0.$\ Since $ Y $ has the $ p $-$ (DPrcP). $\
 \end{proof}

\section{Characterizations of some classes of operators on $  C(\Omega,X)$}
Let us recall  from \cite{bb,du},
every bounded linear operator $  T : C(\Omega,X) \rightarrow Y   $  may be represented
by a vector measure  $m:\Sigma\rightarrow L(X,Y^{\ast\ast}) $ of finite semi-variation  such that
\begin{center}
$ T(f)=\int_{\Omega} f dm, ~~~~f \in C(\Omega,X)$ for every $ f\in C(\Omega,X) .$
\end{center}
 This set function $ m $ is called the representing measure
of $ T. $\ Also,
a bounded linear operator  $ T :
C(\Omega,X) \rightarrow Y $ is called strongly bounded if $  m$ is strongly bounded.\\
Different efforts  have
been done in order to characterize several types of operators in terms of their representing
measure.\
Suppose that $ T:C(\Omega,X)\rightarrow Y $ is a  strongly bounded operator with representing measure $ m:\Sigma\rightarrow L(X,Y) $ and $ \hat{T}:B(\Omega,X)\rightarrow Y $ is its extension.\
The authors  in \cite{bc,bb,bp1,cs,c,g10,g22} have
found that study of $ \hat{T} $  is more suitable than $ T.$\
Our aim in this section is to obtain some characterizations
 of Dunford-Pettis $ p $-convergent operators in terms of their representing
measure.\\
By using the same
argument as in the proof of {\rm (\cite[Theorem 3]{bc})}, we obtain the following
result.\
\begin{thm}\label{t9}
Suppose that $ T : C(\Omega,X) \rightarrow Y $ is a
strongly bounded operator.\ Then, $ T $
is  Dunford-Pettis $ p $-convergent if and only if $ \hat{T} $ is Dunford-Pettis $ p $-convergent.\
\end{thm}
\begin{proof}
Suppose that $ T   $  is strongly bounded  such that $ T$ is Dunford-Pettis $ p $-convergent and
 $ \hat{T} $ is not.\ Therefore,  there exists  a $ p $-Right null sequence $ (y^{\ast}_{n})_{n} $ in $ Y^{\ast} $ and
 a sequence $ (f_{n})_{n} $ in $ B $
 such that $ \vert\langle y^{\ast}_{n},\hat{T}(f_{n}) \rangle\vert=1$
 for all $ n\in \mathbb{N}. $\ Without loss of generality assume  $ \Vert y_{n}^{\ast}  \Vert\leq1 $
for all $ n. $\\
 By using the existence of a control measure for $ m $ and Lusin's theorem, we can find a
compact subset $\Omega_{0} $ of $  \Omega$ such that $ \tilde{m}(\Omega-\Omega_{0}) <\frac{1}{4}$ and
$ g_{n} =f_{n}\vert_{\Omega_{0}}$ is continuous for each $ n\in \mathbb{N}. $\ Let $ H = [g_{n}] $ be the closed linear
subspace spanned by $ (g_{n})_{n} $ in $ C(\Omega_{0},X) .$\ By  {\rm (\cite[Theorem 1]{bc})}, there is an isometric extension operator  $ S : H \rightarrow C(\Omega,X) $ such that the sequence $ h_{n} =S(g_{n})$  is $ p $-Right null in $ C(\Omega, X).$\ It is easy to verify that $  \Vert  T(h_{n})\Vert\geq\frac{1}{2}, $
which is a contradiction with the fact that $ T $ is Dunford-Pettis $ p $-convergent.
\end{proof}

\begin{defn}\label{d2}  A bounded subset $ K $ of  $ X $ is said to be $ p $-Right$ ^{\ast}$ set, if for
every $ p $-Right null sequence $ (x^{\ast}_{n})_{n} $ in $ X^{\ast} $ it follows: $$ \lim_{n} \sup_{x\in K}\vert x_{n}^{\ast}(x) \vert=0.$$\
\end{defn}
\begin{prop}\label{p5}
$ \rm{(i)} $ If $ T \in L(X,Y) ,$ then
$ T^{\ast} $ is Dunford-Pettis $ p $-convergent if and only if  $ T(B_{X}) $ is a $ p $-Right$ ^{\ast} $ subset of $ Y.$ \\
 $ \rm{(ii)} $  $ X^{\ast} $ has the $ p $-$ (DPrcP) $ if and only if every bounded subset of $ X $ is a $ p $-Right$ ^{\ast} $ set.
\end{prop}

\begin{thm}\label{t10}
Suppose that $ T : C(\Omega,X) \rightarrow Y $ is a
strongly bounded operator.\ Then, $ T^{\ast} $
is a Dunford-Pettis $ p $-convergent if and only if $ \hat{T}^{\ast} $ is Dunford-Pettis $ p $-convergent.\
\end{thm}
\begin{proof}
By Proposition \ref{p5} (i), it is enough to show that $ T(B_{0}) $ is a $ p $-Right$ ^{\ast} $ set if and only if $ \hat{T}(B) $ is a $ p $-Right$ ^{\ast} $ set.\ Suppose that  $ T(B_{0}) $ is a $ p $-Right$ ^{\ast} $ set and
$ \hat{T}(B) $ is not a $ p $-Right$ ^{\ast} $ set.\ Therefore there exists $ p $-Right null sequence $ (y^{\ast}_{n}) $ in $ Y^{\ast} $ and   a sequence $ (f_{n}) $ in $ B $ such that $ \vert \langle y^{\ast}_{n},\hat{T}(f_{n})   \rangle \vert=1 $ for each $ n\in \mathbb{N}. $\ By using
the same argument as in the proof of Theorem \ref{t9}, we obtain a sequence $ (h_{n}) _{n}$ in $ B_{0} $ such that $ \vert\langle y^{\ast}_{n}, T(h_{n})\rangle\vert\geq \frac{1}{2} ,$ which is a contradiction, since $ T(B_{0}) $ is a $ p $-Right$ ^{\ast} $ set.\ Similarly, if $ \hat{T}(B) $ is a $ p $-Right$ ^{\ast} $ set, then $ T(B_{0}) $ is a $ p $-Right$ ^{\ast} $ set.
\end{proof}
\begin{cor}\label{c7}
Suppose that  $ T : C(\Omega,X) \rightarrow Y $
is a strongly bounded operator.\ If $ T^{\ast} $
is Dunford-Pettis $ p $-convergent, then for each $ A\in\Sigma, $ the operator $ m(A)^{\ast} :Y^{\ast}\rightarrow X^{\ast}$
 is Dunford-Pettis $ p $-convergent.\
\end{cor}
The proof of the following result is similar to the proof of Theorem \ref{t10}, and will be omitted.
\begin{thm}\label{t11}
Suppose that $ T : C(\Omega,X) \rightarrow  Y^{\ast} $ is a
strongly bounded operator.\ Then $ T^{\ast}\vert_{Y} $
is Dunford-Pettis $ p $-convergent if and only if $ \hat{T}^{\ast}\vert_{Y} $ is  Dunford-Pettis $ p$-convergent.
\end{thm}

\begin{cor}\label{c8}
Suppose that  $  T:C(\Omega,X) \rightarrow  Y^{\ast}$
is a strongly bounded operator.\ If $ T^{\ast}\vert_{Y} $ is  Dunford-Pettis $ p $-convergent, then for each $ A\in \Sigma, $ $ m(A)^{\ast}\vert_{Y} $ is   Dunford-Pettis $ p $-convergent.
\end{cor}

Let us recall from \cite{g22}, that  if $ T : C(\Omega,X) \rightarrow Y $ is a bounded linear operator, $ \overline{\Omega} $ is a metrizable compact space, and $ \pi:\Omega\rightarrow  \overline{\Omega} $ a continuous map which is onto, we will call $ \overline{\Omega} $ a quotient of $ \Omega. $\ The map $\overline{\pi}:C( \overline{\Omega})\rightarrow C(\Omega) $ given by $ \overline{\pi}( \overline{f})= \overline{f}\circ \pi $ defines an isometric embedding of $ C( \overline{\Omega}) $ into $ C(\Omega). $\ Let $  \overline{T} : C(\overline{\Omega},X) \rightarrow Y $ be the operator defined by $ \overline{T}(\overline{f})=T(\overline{f}\circ \pi) $ where $ \overline{f}\in C(\overline{\Omega},X) $ and $ \pi:\Omega\rightarrow \overline{\Omega} $ is the canonical mapping.\\
By using the same
argument as in the proof of {\rm (\cite[Lemma 18]{g22})}, we obtain the following result.
\begin{thm}\label{t12}
$ \rm{(i)} $ Suppose that  $ T : C(\Omega,X) \rightarrow  Y $ is a bounded linear operator.\ Then,
$ T^{\ast} $ is  Dunford-Pettis $ p $-convergent if and only if for each metrizable quotient $ \overline{\Omega}$
of $ \Omega, $ $ \overline{T}^{\ast} $
is Dunford-Pettis $ p $-convergent.\\
$ \rm{(ii)} $ Suppose that $ T : C(\Omega,X) \rightarrow Y^{\ast} $ is a bounded linear operator.\ Then,
$ T^{\ast}\vert_{Y} $ is  Dunford-Pettis $ p $-convergent if and only if for each metrizable quotient $ \overline{\Omega} $ of $\Omega, $
$ \overline{T}^{\ast}\vert_{Y} $ is Dunford-Pettis $ p $-convergent.
\end{thm}
\begin{proof}
We will only consider the part (i).\ The proof of (ii) is similar.\
Suppose that  adjoint $ T : C(\Omega,X) \rightarrow Y  $ is Dunford-Pettis $ p $-convergent
 and
$\overline{\Omega} $
is a metrizable quotient space of $ \Omega. $\ Then $ \overline{T}^{\ast} $ is Dunford-Pettis $ p $-convergent.\\
Conversely, let $ T : C(\Omega,X) \rightarrow Y $ be a bounded linear operator and let $ (f_{n})_{n} $ be a sequence
in $ B_{0}. $\ It is known (see \cite{bb}) that there exists a  metrizable quotient space $ \overline{\Omega} $ of $  \Omega$ and a sequence $ (\overline{f_{n}}) $ in unit ball $ C(\overline{\Omega} ,X) $ such
that  $ (\overline{f_{n}})(\pi (t) )=f_{n}(t)$ for all $ t \in \Omega $ and $ n\in \mathbb{N} .$\ Define $\overline{T}  :C(\overline{\Omega},X) \rightarrow Y$ by $ \overline{T}(\overline{f}) =T(\overline{f}\circ \pi),$
 where$  $ is the $ \pi:\Omega\rightarrow \overline{\Omega} $ canonical mapping.\ By the hypothesis,
$ \overline{T}^{\ast} $ is Dunford-Pettis $  p$-convergent.\ Therefore $\lbrace T(f_{n}):n\in \mathbb{N}\rbrace= \lbrace \overline{T}(\overline{f_{n}}):n\in \mathbb{N}\rbrace$  is a $ p $-Right$ ^{\ast} $  set.\ Hence, the part $ \rm{(i)} $ of Proposition \ref{p5} implies that
$ T^{\ast} $ is  Dunford-Pettis $p$-convergent.
\end{proof}
The proofs of the following Lemma is similar to that of {\rm(\cite[Lemma 21]{g22})} and
 and will be omitted.
\begin{lem}\label{l4}
 Let $  K$ be a bounded subset of $ X. $\ If for each $ \varepsilon >0 $ there is a $ p $-Right$ ^{\ast} $ subset $ K_{\varepsilon} $ of $ X $ such that   $ K\subseteq K_{\varepsilon}+\varepsilon B_{X} ,$  then $ K $ is a $ p $-Right$ ^{\ast} $ set.\
\end{lem}
A topological space $ \Omega $ is called dispersed (or scattered), if every nonempty
closed subset of $  \Omega$ has an isolated point \cite{s}.\ By using the same
argument as in the proof of {\rm (\cite[Theorem 22]{g22})}, we obtain the following
result.
\begin{thm}\label{t13}
$ \rm{(i)} $ Suppose that $  \Omega$ is a dispersed compact Hausdorff space
and $ T : C(\Omega,X) \rightarrow Y $ is a strongly bounded operator.\
Then $ T^{\ast}  $ is Dunford-Pettis $ p $-convergent if and only if $ m(A) ^{\ast}: Y^{\ast}\rightarrow X^{\ast}$
is Dunford-Pettis $ p $-convergent, for each $ A \in \Sigma. $\\
$ \rm{(ii)} $  Let $  \Omega$ be a dispersed compact Hausdorff space
and $  T : C(\Omega,X)\rightarrow Y^{\ast} $ be a strongly bounded operator.\
Then $T^{\ast}\vert_{Y} : Y \rightarrow C(\Omega,X)^{\ast}  $ is Dunford-Pettis $ p $-convergent if and only if
$ m(A)^{\ast}\vert_{Y} : Y \rightarrow X^{\ast} $ is Dunford-Pettis $ p $-convergent, for each $A \in \Sigma,  $
\end{thm}
\begin{proof}
We will only consider the part (i).\ The proof of (ii) is similar.\
Suppose  that $  T : C(\Omega,X) \rightarrow Y $ is strongly bounded such that  $ T^{\ast} $ is  Dunford-Pettis $ p $-convergent, then for each $ A \in \Sigma, $ Corollary \ref{c7}  implies that $ m(A)^{\ast}:Y^{\ast}\rightarrow X^{\ast} $ is a Dunford-Pettis $ p $-convergent operator. \\
Conversely, suppose that $ m(A)^{\ast}:Y^{\ast}\rightarrow X^{\ast} $ is Dunford-Pettis $ p $-convergent for each $ A \in \Sigma. $\
Since a quotient space of a dispersed
space is dispersed (see {\rm (\cite[8. 5. 3]{s})}), by using Lemmas 18 and 20 in \cite{g22}, we can suppose without loss of generality that
$ \Omega $ is metrizable and so, $ \Omega $ is countable (see {\rm (\cite[8. 5. 5]{s})}).\ Suppose that  $ \Omega=\lbrace t_{i}:i\in \mathbb{N}    \rbrace $  and $ (f_{n}) $ is a sequence in $ B_{0}. $\ For each $  i\in \mathbb{N}$ the set $ \lbrace f_{n}(t_{i}):i\in \mathbb{N}    \rbrace  $  is bounded in $ X. $\
Therefore Proposition \ref{p5} (i) implies that, the set $ H_{i} =\lbrace m(\lbrace  t_{i}\rbrace) (f_{n}(t_{i})): n\in \mathbb{N}    \rbrace$
is a $  p$-Right$ ^{\ast} $ set, for each $ i\in \mathbb{N}. $\ Now, let  $ A_{i} =\lbrace t_{j}:j>i   \rbrace, i\in \mathbb{N}$ and $ \varepsilon>0. $\ Since $  m$ is strongly bounded,
there is a $ k\in \mathbb{N} $ such that $  \title{m}(A_{k})<\varepsilon.$\ Hence, for each $ n\in \mathbb{N}, $
\begin{center}
$ T(f_{n}) =\int_{\Omega} f_{n}dm=\displaystyle\sum_{i=1}^{k}m(\lbrace t_{i} \rbrace)(f_{n}(t_{i}))+\int_{A_{k}} f_{n}dm.$
\end{center}
Further, $\Vert  \int_{A_{k}} f_{n}dm  \Vert < \tilde{m}(A_{k})<\varepsilon.$\ Therefore
\begin{center}
$  T(f_{n}) \in H_{1} +H_{2} +\cdot\cdot\cdot+H_{k}+\varepsilon B_{Y}.$
\end{center}
Since $ H_{1} +H_{2} +\cdot\cdot\cdot+H_{k} $ is a $ p $-Right$ ^{\ast} $ set, by Lemma \ref{l4},
 $  \lbrace  T(f_{n}) :n \in \mathbb{N} \rbrace$ is
a $ p $-Right$ ^{\ast} $ set.\ Thus, an application of Proposition \ref{p5} (i) gives that
 $ T^{\ast} $ is Dunford-Pettis $ p $-convergent.
\end{proof}

\end{document}